\documentclass[draft]{amsart}

\usepackage{amssymb}
\usepackage{url}
\usepackage{amscd}
\usepackage{amsmath}
\usepackage{amssymb,amsthm,amsfonts,amsbsy,latexsym,dsfont}

\theoremstyle{plain}
\newtheorem{theorem}{Theorem}
\newtheorem{conj}{Conjecture}

\newtheorem{lemma}[theorem]{Lemma}

\theoremstyle{remark}

\providecommand{\1}{\mathds{1}}

\begin{document}

\date{Received April 1, 2020} 

\title{On stability of the Erd\H{o}s-Rademacher Problem}


\author{J\'ozsef Balogh}
\address{Department of Mathematics, University of Illinois at Urbana-Champaign, Urbana, Illinois 61801, USA, and Moscow Institute of Physics and Technology, Russian Federation.}
\thanks{Research of the first author is partially supported by NSF Grant DMS-1764123, Arnold O. Beckman Research Award (UIUC Campus Research Board RB 18132) and the Langan Scholar Fund (UIUC)}

\author{Felix Christian Clemen}
\address{Department of Mathematics, University of Illinois at Urbana-Champaign, Urbana, Illinois 61801, USA.}

\email{jobal@illinois.edu, fclemen2@illinois.edu.}

\maketitle

\begin{abstract}
Mantel's theorem states that every $n$-vertex graph with $\lfloor \frac{n^2}{4} \rfloor +t$ 
edges, where $t>0$, contains a triangle. The problem of determining the minimum number of triangles in such a graph is usually referred to as the Erd\H{o}s-Rademacher problem. Lov\'asz and Simonovits proved that there are at least $t\lfloor n/2 \rfloor$ triangles in each of those graphs. Katona and Xiao considered the same problem under the additional condition that there are no $s-1$ vertices covering all triangles. They settled the case $t=1$ and $s=2$. Solving their conjecture, we determine the minimum number of triangles for every fixed pair of $s$ and $t$, when $n$ is sufficiently large. Additionally, solving another conjecture of Katona and Xiao, we extend the theory for considering cliques instead of triangles.
\end{abstract}
\section{Introduction}
A classical result in extremal combinatorics is Mantel's theorem~\cite{Mantel}. It says that every $n$-vertex graph with at least $\lfloor \frac{n^2}{4} \rfloor +1$ edges contains a triangle. There have been various extensions of Mantel's theorem.  Tur\'an~\cite{Turanstheorem} generalized it for cliques: Every $n$-vertex graph with at least $t_r(n)+1$ edges contains a clique of size $r+1$, where $t_r(n)$ is the number of edges of the complete balanced $r$-partite graph on $n$ vertices. A result of Rademacher, see \cite{MR81469}, says that every graph on $n$ vertices with $\lfloor \frac{n^2}{4} \rfloor +1$ edges contains at least $\lfloor \frac{n}{2} \rfloor$ triangles and this is best possible. Erd\H{o}s~\cite{MR81469,MR0137661} conjectured that an $n$-vertex graph with $\lfloor \frac{n^2}{4} \rfloor +t$ edges for $t<\frac{n}{2}$ contains at least $t\lfloor \frac{n}{2}\rfloor$ triangles. This was proven by Lov\'asz and Simonovits~\cite{Lovasz}. Xiao and Katona~\cite{Katona} proved a stability variant of the Lov\'asz and Simonovits result for $t=1$: If there is no vertex contained in all triangles, then there are at least $n-2$ triangles in $G$. We will prove two extensions of this statement, one verifying a conjecture by Xiao and Katona~\cite{Katona} and the other proving a slight modification of another conjecture by Xiao and Katona~\cite{Katona}. 

For a graph $G$ denote $e(G)$ the number of edges of $G$ and $T_r(G)$ the number of copies of $K_{r}$ in $G$. A $K_r$-\emph{covering set} in $V(G)$ is a vertex set that contains at least one vertex of every copy of $K_r$ in $G$. The $r$-\emph{clique covering number} $\tau_{r}(G)$ is the size of the smallest $K_r$-covering set. A \emph{triangle covering set} is a $K_3$-covering set and the  \emph{triangle covering number} is the $3$-clique covering number. Given a vertex partition $V(G)=V_1\cup V_2 \cup \ldots \cup V_r$, we call an edge $e\in E(G)$ a \emph{class-edge} if $e\in E(G[V_i])$ for some $i$ and a \emph{cross-edge} otherwise. A set of two vertices $\{x,y\}$ is a \emph{missing cross-edge} if $xy\not\in E(G)$ and $x\in V_i,y\in V_j$ for some $i\neq j$. Denote $e_G(V_1,V_2,\ldots,V_r)$ the number of cross-edges and  $e_G^c(V_1,V_2,\ldots,V_r)$ the number of missing cross-edges. We drop the index $G$ and just write $e(V_1,V_2,\ldots,V_r)$ and $e^c(V_1,V_2,\ldots,V_r)$, respectively, if $G$ is clear from context.
\begin{conj}[Xiao, Katona~\cite{Katona}]
\label{Katona1}
Let $t,s\in \mathbb{N}$ such that $0<t<s$. Suppose the graph $G$ has $n$ vertices and $\left\lfloor \frac{n^2}{4}\right\rfloor+t$ edges satisfying $\tau_{3}(G)\geq s$ and $n\geq n(s,t)$ is large. Then $G$ contains at least 
    \begin{align*}
        (s-1) \left\lfloor \frac{n}{2}\right\rfloor + \left \lceil \frac{n}{2}\right\rceil -2 (s-t)
    \end{align*}
    many triangles.
\end{conj}
Xiao and Katona~\cite{Katona} proved their conjecture for $t=1$ and $s=2$ and that it is best possible. However, their conjecture in full generality holds only up to an additive constant. We will determine precisely the minimum number of triangles a graph on $n$ vertices with $\lfloor n^2/4 \rfloor+t$ edges and $\tau_3(G)\geq s$ can have. Depending on $s,t$ and $n$ one of the following two constructions will be an extremal example. \\
{\bf Construction 1:} Let $G_1$ be the graph on $n$ vertices, where $V(G_1)=A \cup B$ with 
\begin{align*}
|A|=\left\lceil \frac{n}{2} \right\rceil  +a,\quad  |B|=\left\lfloor \frac{n}{2}\right\rfloor -a,
\end{align*}
where $a$ is chosen later to minimize the number of triangles in this construction. Pick $2(s-1)$ vertices 
 $x_1,y_1,\ldots,x_{s-1},y_{s-1}$ in $A$ and two vertices $u_1,u_2$ in $B$. Add the edges $\{x_i,y_i\}$ and $\{u_1,u_2\}$ to $K_{|A|,|B|}$ and delete the edges $\{u_1,x_1\},\ldots, \{u_1,x_\alpha \}$, where $\alpha:=s-t-a^2-\1_{\{2 \nmid n\}}a$. 
 Then, $G_1$ has $\lfloor \frac{n^2}{4} \rfloor +t$ edges, triangle covering number $\tau_{3}(G_1)=s$ and
\begin{align}
\label{constr1}
    T_3(G_1)&=  (s-1) \left(\left\lfloor \frac{n}{2}\right\rfloor-a\right) + \left( \left\lceil \frac{n}{2}\right\rceil +a\right)-2\left(s-t-a^2-\1_{\{2 \nmid n\}}a\right).
\end{align}
Now, choose $a\geq 0$ satisfying $\alpha\geq0$ and minimizing \eqref{constr1}. Note that $G_1$ is similar to the construction of Xiao and Katona~\cite{Katona} resulting in Conjecture~\ref{Katona1}, except that the class sizes are different in some cases. \\ 
{\bf Construction 2:} Let $G_2$ be the graph on $n$ vertices, where $V(G_2)=A \cup B$ with 
\begin{align*}|A|=\left\lceil \frac{n}{2} \right\rceil + a, \quad |B|=\left\lfloor \frac{n}{2}  \right\rfloor -a,
\end{align*}
where $a$ is chosen later to minimize the number of triangles in this construction.
Pick $2s$ vertices 
 $x_1,y_1,\ldots,x_{s},y_{s} \in A$ and one vertex $u\in B$. Add the edges $\{x_i,y_i\}$  to $K_{|A|,|B|}$ and delete the  edges $\{u,x_1\},\ldots, \{u,x_\alpha \}$, where $\alpha:=s-t-a^2-\1_{\{2 \nmid n \}}a$. The graph $G_2$ has $\lfloor \frac{n^2}{4} \rfloor +t$ edges, triangle covering number $\tau_{3}(G_2)=s$ and
\begin{align}
\label{constr2}
    T_3(G_2)=   s \left(\left\lfloor \frac{n}{2}\right\rfloor-a\right)-\left(s-t-a^2-\1_{\{2 \nmid n \}}a \right).
\end{align}
Now, choose $a\geq 0$ satisfying $\alpha\geq0$ and minimizing \eqref{constr2}. Our first main result is the following.
\begin{theorem}
\label{trianglemain}
    Let $t,s\in \mathbb{N}$ such that $0<t<s$. Suppose the graph $G$ has $n$ vertices and $\left\lfloor \frac{n^2}{4}\right\rfloor+t$ edges, $\tau_{3}(G)\geq s$ and $n\geq n(s,t)$ is large. Then $G$ contains at least $\min \{T_{3}(G_1),T_3(G_2)\}$ triangles.
\end{theorem}    
Note that both $G_1$ and $G_2$ achieve this minimum for some values of $s$ and $t$.
Xiao and Katona~\cite{Katona} also conjectured what the minimum number of copies of $K_{r+1}$ in a graph $G$ on $n$ vertices with $t_{r-1}(n)+1$ edges and $\tau_{r+1}(G)\geq2$ could be. They conjectured the following graph to be an extremal example.\\
{\bf Construction 3:} 
The graph $G_3$ is constructed as follows from the complete $r$-partite graph with vertex partition $V(G_3)=V_1\cup V_2 \cup \ldots \cup V_r$ where $|V_1|\geq|V_2|\geq \ldots \geq |V_r|$ and $|V_1|-|V_r|\leq 1$. Let $v_1,v_2\in V_1$ and $u_1,u_2 \in V_2$. Remove the edge $\{v_1,u_1\}$ and add the edges $\{v_1,v_2 \}$ and $\{u_1,u_2 \}$. The graph $G_3$ satisfies $\tau_{r+1}(G_3)=2$, $e(G_3)=t_{r}(n)+1$ and 
\begin{align*}
    T_{r+1}(G_3)=(|V_1|+|V_2|-2) \prod_{i=3}^r|V_i|.
\end{align*}
We will verify Xiao and Katona's~\cite{Katona} conjecture. 
\begin{theorem}
\label{conjKatona2}
  Let $r\in\mathbb{N}$ and $n\geq n(r)$ be a sufficiently large integer. If a graph $G$ on $n$ vertices has $t_{r}(n)+1$ edges and the copies of $K_{r+1}$ have an empty intersection, then the number of copies of $K_{r+1}$ is at least as many as in $G_3$.
\end{theorem}
The key ingredient for proving Theorems~\ref{trianglemain} and \ref{conjKatona2} is the following general structural result.  
\begin{theorem}
\label{propositionmainclique}
Let $r,t,s\in \mathbb{N}$ with $0<t<s$ and let $n\geq n(s,t,r)$ be sufficiently large. Let $G$ be an $n$-vertex graph with $t_{r}(n)+t$ edges satisfying $\tau_{r+1}(G)\geq s$. Denote $\mathcal{H}$ the family of graphs $H$ on $n$ vertices with $t_r(n)+t$ edges such that there exists a vertex partition $V(H)=V_1\cup V_2 \cup \ldots \cup V_r$ satisfying 
\begin{itemize}
    \item the class-edges form a matching of size $s$,
    \item $e_H^c(V_1,\ldots,V_r)\leq s-t$,
    \item all missing cross-edges are incident to one of the class-edges,
    \item if the class-edges are not all in the same class, then each missing cross-edge is incident to class-edges on both endpoints. 
\end{itemize}
Then $T_{r+1}(G)\geq T_{r+1}(H)$ for some $H\in \mathcal{H}$.
\end{theorem}
We will observe (Lemma~\ref{classsizesandcounting}) that 
$|V_i|=n/r+O(1)$ for all $1\leq i\leq r$. Hence, the family $\mathcal{H}$ is small. An optimization argument will reduce the family $\mathcal{H}$ to $G_1$ and $G_2$ in the case $r=3$, and to $G_3$ in the case $r> 3,s=2,t=1$. A similar cleaning argument could potentially  be done in the general case, however the computational effort needed seems not worth the outcome. Also, it is likely that for some parameters $r,t,s$, the extremal family realizing the minimum number of cliques of size $r+1$ contains several graphs.

For $t\geq s$ we know from a result of Erd\H{o}s~\cite{MR0252253} that the number of cliques of size $r+1$ is at least $(1+o(1))t \left(\frac{n}{r}\right)^{r-1}$ and the graph $G_4$ constructed by adding a matching of size $t$ to one of the classes of the complete balanced $r$-partite graph on $n$ vertices satisfies $\tau_{r+1}(G_4)=t$ and $T_{r+1}(G_4)=(1+o(1))t \left(\frac{n}{r}\right)^{r-1}$. Hence this problem is interesting only when $0<t<s$. 

Our result can be extended to consider $s$ and $t$ as functions of $n$. In particular, in the triangle case $r=3$, our proof allows $s$ to be linear in $n$. However, the methods we use, especially our main tool Theorem~\ref{cliquesupersat}, will not give us the entire range of $s$ where our theorem should hold. Thus, we do not attempt to optimize our proof with respect to the dependency of $n$ and $s$.

Our main tool is a 
stability-supersaturation result by Balogh, Bushaw, Collares, Liu, Morris and Sharifzadeh~\cite{baloghcliques}, which extends on the Erd\H{o}s-Simonovits stability theorem~\cite{Erdossimonovits} and F\"uredi's~\cite{Furedi} proof's of it. 

We would like to point out that Liu and Mubayi~\cite{Liumubayi} independently obtained similar results to ours.

Our paper is organized as follows. In Section~\ref{sec:propos} we prove Theorem~\ref{propositionmainclique}, in Section~\ref{sec:cor1} we conclude Theorem~\ref{trianglemain} and in Section~\ref{sec:cor2} we conclude Theorem~\ref{conjKatona2}.

\section{\bf Proof of Theorem~\ref{propositionmainclique}}
\label{sec:propos}
The well-known Tur\'an's theorem~\cite{Turanstheorem} determines the maximum number of edges in a $K_{r+1}$-free graph to be the number of edges $t_r(n)$ of the complete balanced $r$-partite graph. We will make use of the following bound on $t_r(n)$.
\begin{align}
\label{Turanedges}
\frac{n^2}{2} \left(1-\frac{1}{r} \right) - \frac{r}{2} \leq t_r(n) \leq \frac{n^2}{2} \left(1-\frac{1}{r}\right).
\end{align}
The classical Erd\H{o}s-Simonovits stability theorem~\cite{Erdossimonovits} asserts that any $n$-vertex $K_{r+1}$-free graph with almost $t_r(n)$ edges is close to the complete balanced $r$-partite graph. 
There are many different quantitative versions~\cite{Makingrpartite,KRS,RS} of this. A standard calculation shows that $n$-vertex graphs with almost $t_r(n)$ edges and a vertex partition with few class-edges need to have almost balanced class sizes. 
\begin{lemma}
\label{classsizesandcounting}
Let $s,n\in \mathbb{N}$, $t\in \mathbb{Z}$ and $G$ a graph on $n$ vertices and $t_r(n)+t$ edges with vertex partition $V(G)=V_1\cup V_2\cup \ldots \cup V_r$ such that the number of class-edges is $s$. Then
$|V_i|=\frac{n}{r}+O(1)$ for all $1\leq i \leq r$.
\end{lemma}
\begin{proof}
Without loss of generality we can assume that $|V_1|\geq |V_2|\geq \ldots \geq|V_r|$. Now, let $x=|V_1|-\frac{n}{r}$. Then,
\begin{align}
\label{classsizessquared}
\nonumber
    \sum_{i=1}^r |V_i|^2 &=  \left(\frac{n}{r}+x\right)^2 + \sum_{i=2}^r |V_i|^2  \geq \left( \frac{n}{r}+ x \right)^2 + \frac{\left(  \sum_{i=2}^r |V_i|\right)^2}{r-1} \\
    &\geq \left( \frac{n}{r}+x \right)^2 + \frac{\left(n \left(1-\frac{1}{r} \right)-x \right)^2}{r-1} 
\geq \frac{n^2}{r} + x^2. 
\end{align}
Thus, we can conclude
\begin{align*}
    & \quad \ t_r(n)+t=e(G)=\sum_{i=1}^re(G[V_i])+ e(V_1,V_2,\ldots,V_r)
    \leq s + \sum_{1\leq i<j\leq r}|V_i||V_j|\\
    &= s+ \frac{1}{2}\left(n^2-\sum_{i=1}^r|V_i|^2\right) \leq s+ \frac{n^2}{2}\left(1-\frac{1}{r}\right) -\frac{x^2}{2} \leq t_r(n)+\frac{r}{2}+s-x^2,
\end{align*}
implying $x=O(1)$. Note that in the second to last inequality we used \eqref{classsizessquared} and in the last inequality we used \eqref{Turanedges}. We conclude $|V_1|=\frac{n}{r}+O(1)$. Similarly, we can conclude $|V_r|=\frac{n}{r}+O(1)$.
\end{proof}
\begin{lemma}
\label{classsizesandcounting2}
Let $s,n\in \mathbb{N}$, $t\in \mathbb{Z}$ and $G$ a graph on $n$ vertices and $t_r(n)+t$ edges with vertex partition $V(G)=V_1\cup V_2\cup \ldots \cup V_r$ such that the number of class-edges is $s$. Then
\begin{align*}
T_{r+1}(G)=(1+o(1))s \left(\frac{n}{r}\right)^{r-1}.
\end{align*}
\end{lemma}
\begin{proof}
Any copy of $K_{r+1}$ needs to contain at least one class-edge. There are exactly $(\frac{n}{r}+O(1))^{r-1}$ copies of $K_{r+1}$ containing one given class-edge but no others, since the other $r-1$ vertices need to be chosen from different classes. Since the number of missing cross-edges is $O(1)$, they do not have an impact in this counting. There are $s$ class-edges and thus we conclude that the number of copies of $K_{r+1}$ containing exactly one class-edge is 
\begin{align*}
    (1+o(1))s\left(\frac{n}{r}\right)^{r-1}.
\end{align*}
The number of copies containing at least two class-edges is $O(n^{r-2})$, because there are $O(1)$ ways to chose three vertices among the vertices being incident to the class-edges and at most $O(n^{r-2})$ ways to chose the remaining $r-2$ vertices. We conclude
\begin{equation*}
T_{r+1}(G)=(1+o(1))s \left(\frac{n}{r}\right)^{r-1}.
\qedhere
\end{equation*}
\end{proof}

 We say that a graph $G$ is $x$-\emph{far from being $r$-partite} if $\chi(G')>r$ for every subgraph $G'\subset G$ with $e(G')>e(G)-x$. Balogh, Bushaw, Collares, Liu, Morris and Sharifzadeh~\cite{baloghcliques} proved that graphs which are $x$-far from being $r$-partite, contain many cliques of size $r+1$.
 \begin{theorem}[\cite{baloghcliques}]
\label{cliquesupersat}
For every $n,x,r\in \mathbb{N}$, the following holds. Every graph $G$ on $n$ vertices which is $x$-far from being $r$-partite contains at least 
\begin{align*}
    \frac{n^{r-1}}{e^{2r}r!}\left( e(G) +x - \left( 1-\frac{1}{r} \right)\frac{n^2}{2}\right)
\end{align*}
copies of $K_{r+1}$. 
\end{theorem}
With this tool in hand, we now prove Theorem~\ref{propositionmainclique}.
\begin{proof}[\bf Proof of Theorem~\ref{propositionmainclique}]
Let $G$ be a graph on $n$ vertices and $t_{r}(n)+t$ edges such that $\tau_{r+1}(G)\geq s$ and $n$ is large enough. If $G$ is $2sr!e^{2r}$-far from being $r$-partite, then by applying Theorem~\ref{cliquesupersat}, the number of copies of $K_{r+1}$ in $G$ is at least 
    \begin{align*}
        T_{r+1}(G)\geq \frac{n^{r-1}}{r!e^{2r}}\left( t_{r}(n) +t+ 2sr!e^{2r} - \left( 1-\frac{1}{r}\right)\frac{n^2}{2}\right)\geq  sn^{r-1}, 
    \end{align*}
where we used \eqref{Turanedges} to lower bound $t_r(n)$. By Lemma~\ref{classsizesandcounting2}, $T_{r+1}(H)=(1+o(1))s(\frac{n}{r})^{r-1}$ for every $H\in \mathcal{H}$, and thus $T_{r+1}(H)\leq T_{r+1}(G)$. Hence, we can assume that $G$ is not $2sr!e^{2r}$-far from being $r$-partite. Then there exists a subgraph $G'\subset G$ with $\chi(G')\leq r$ and 
\begin{align*}
e(G')>e(G)-2sr!e^{2r}s > t_r(n) - 2sr!e^{2r}.
\end{align*}
Let $V(G')=V_1\cup V_2 \cup \ldots \cup V_r$ be a partition such that $V_1,V_2,\ldots,V_r$ are independent sets in $G'$. By Lemma~\ref{classsizesandcounting}, the class sizes are roughly balanced, that is $|V_i|= \frac{n}{r}+ O(1)$ for all $1\leq i \leq r$.

Since $G'\subset G$, and $e(G')>e(G)-2sr!e^{2r}$, the vertex partition $V(G)=V_1\cup V_2 \cup \ldots \cup V_r$ contains at most $2sr!e^{2r}$ class-edges in $G$. If the number of class-edges in $G$ is less than $s$, then there is a $K_{r+1}$-covering set of size at most $s-1$, contradicting $\tau_{r+1}(G)\geq s$. If the number of class-edges is more than $s$, then by Lemma~\ref{classsizesandcounting2}
\begin{align*}
    T_{r+1}(G)=(1+o(1))(s+1)\left(\frac{n}{r}\right)^{r-1}\geq (1+o(1))s\left(\frac{n}{r}\right)^{r-1}=T_{r+1}(H)
\end{align*}
for every $H\in \mathcal{H}$. Therefore, we can assume that the number of class-edges in $G$ is exactly $s$. These $s$ edges need to form a matching, otherwise there is a $K_{r+1}$-covering set of size $s-1$, contradicting $\tau_{r+1}(G)\geq s$. Further, $e^c(V_1,\ldots, V_r)\leq s-t$, because otherwise 
\begin{align*}
    e(G) &= e(V_1,V_2,\ldots,V_r)+\sum_{i=1}^r e(G[V_i])\leq t_{r+1}(n) -e^c(V_1,V_2,\ldots,V_r)+s \\ &< t_{r+1}(n) + t =e(G).  
\end{align*}
Next, we assume that there is a missing cross-edge $uv\not\in E(G)$ not incident to any of the class-edges. Take a cross-edge $xy\in E(G)$ which is incident to exactly one class-edge. We will show that $T_{r+1}(G+uv-xy) < T_{r+1}(G)$. 
Adding $uv$ to $G$ increases the number of cliques of size $r+1$ by at most $sn^{r-3}$,
because for each of the $s$ class-edges $e$ the number of cliques of size $r+1$ containing $e$ and $uv$ is increased by at most $n^{r-3}$.
Removing $xy$ from $G+uv$ decreases the number of cliques of size $r+1$ by at least 
\begin{align*}
    \min (|V_i|-s)^{r-2}\geq \left(\frac{n}{r}+O(1)-s \right)^{r-2}=\Omega(n^{r-2}),
\end{align*}    
where we used that in each class $V_i$ the number of vertices incident to no missing cross-edge is at least $|V_i|-s$. 
Thus, performing this edge flip decreases the number of cliques of size $r+1$, i.e. $T_{r+1}(G+uv-xy)< T_{r+1}(G)$. We repeat this process until we end up with a graph $\bar{G}$ such that $T_{r+1}(\bar{G})\leq T_{r+1}(G)$ and every missing cross-edge is incident to a class-edge.\\
If $\bar{G}$ has all class-edges in the same class, then $\bar{G}\in\mathcal{H}$. Thus, we can assume that $\bar{G}$ has the property that not all class-edges are in the same class and there is a missing cross-edge not adjacent to class-edges on both ends. Denote $G'$ the graph constructed from $\bar{G}$ by adding all missing cross-edges. This graph satisfies 
\begin{align*}
    T_{r+1}(G')&\leq T_{r+1}(\bar{G})+ \left(2e_{\bar{G}}^c(V_1,V_2,\ldots,V_r)-1\right)\left( \frac{n}{r}+O(1)\right)^{r-2} \\ &+s e_{\bar{G}}^c(V_1,V_2,\ldots,V_r) n^{r-3} \\
    &\leq T_{r+1}(G)+ (1+o(1))\left(2e_{\bar{G}}^c(V_1,V_2,\ldots,V_r)-1\right) \left(\frac{n}{r}\right)^{r-2},
\end{align*}
because for each class-edge $e$ and each added edge $e'$, the number of cliques of size $r+1$ containing $e$ and $e'$ increases by at most $\left( \frac{n}{r}+O(1)\right)^{r-2}$
if $e$ and $e'$ share a vertex and by at most $n^{r-3}$ if $e$ and $e'$ are disjoint. 

There is a set $M$ of $e_{\bar{G}}^c(V_1,V_2,\ldots,V_r)\leq s-1$ cross-edges such that each $e\in M$ is incident to class-edges on both endpoints and every class-edge is incident to edges from $M$ on at most one endpoint. Denote $H$ the graph constructed from $G'$ by removing those edges. This graph satisfies $H\in \mathcal{H}$ and
\begin{align*}
    T_{r+1}(H)&\leq T_{r+1}(G')-(2+o(1))e_{\bar{G}}^c(V_1,V_2,\ldots,V_r) \left( \frac{n}{r}+O(1)-s\right)^{r-2}\\
    &\leq T_{r+1}(G),
\end{align*}
completing the proof.
\end{proof}


\section{\bf Proof of Theorem~\ref{trianglemain}}
\label{sec:cor1}

\begin{proof}[\bf Proof of Theorem~\ref{trianglemain}]
By Theorem~\ref{propositionmainclique}, we can assume that $G$ has a vertex partition $V(G)=A \cup B$ where $|A|\geq |B|$ with exactly $s$ class-edges forming a matching. Let $|A|=\lceil n/2 \rceil +a, |B|=\lfloor n/2 \rfloor -a$ for some $a\geq 0$. We have
\begin{align*}
    \left\lfloor \frac{n^2}{4}\right\rfloor +t& =e(G)=|A||B|-e^c(A,B)+s \\
    &=\left( \left\lceil \frac{n}{2}\right\rceil +a\right)\left( \left\lfloor \frac{n}{2}\right\rfloor -a \right)-e^c(A,B)+s,
\end{align*}
and thus the number of missing cross-edges is
\begin{align*}
    e^c(A,B)=s-t-a^2-\1_{\{2\nmid n\}}.
\end{align*}
Let $M_A$ be the matching of class-edges inside $A$ and $M_B$ be the matching of class-edges inside $B$. Denote $G'$ the graph constructed from $G$ by adding all missing cross-edges $e^c(A,B)$. The number of triangles in this graph is $|M_A||B|+|M_B||A|$. \\
{\bf Case 1: $|M_B|>0$}. \\
The number of triangles in $G$ is at least $|M_A||B|+|M_B||A|-2e^c(A,B)$, because each missing cross-edge is in at most two triangles in $G'$. Since $|A|\geq|B|$, we get
\begin{align*}
    T_{3}(G)&\geq (s-1) \left(\left\lfloor \frac{n}{2}\right\rfloor-a\right) + \left( \left\lceil \frac{n}{2}\right\rceil +a\right)-2e^c(A,B) \\
&= (s-1) \left( \left\lfloor \frac{n}{2} \right\rfloor -a \right) + \left( \left\lceil \frac{n}{2}  \right\rceil +a \right) -2\left(s-t-a^2-\1_{\{2\nmid n\}}a\right) \\ &\geq T_3(G_1).
\end{align*}
{\bf Case 2: $|M_B|=0$}.\\
The number of triangles in $G$ is at least $|M_A||B|-e^c(A,B)$, because each missing cross-edge is in at most one triangle in $G'$. Therefore
\begin{align*}
    T_3(G)&\geq |M_A||B|-e^c(A,B)= s \left(\left\lfloor \frac{n}{2}\right\rfloor-a\right)-\left(s-t-a^2-\1_{\{2 \nmid n\}a}\right)  \\ 
    &\geq T_3(G_2).
    \hfill \qedhere
\end{align*}
\end{proof}

\newcommand{\Mod}[1]{\ (\mathrm{mod}\ #1)}

\section{\bf Proof of Theorem~\ref{conjKatona2}}
\label{sec:cor2}
Let $\ell=n \mod r$ and $m$ be an integer such that $n=rm+\ell$. The number of cliques of size $r+1$ in $G_3$ is 
\begin{align*}
    T_{r+1}(G_3)&=(|V_1|+|V_2|-2) \prod_{i=3}^r|V_i| =\begin{cases}
    \left(2m-2\right)   m^{r-2}&, \text{ iff } \ell=0, \\
     \left(2m-1 \right)   m^{r-2}&, \text{ iff } \ell=1, \\
    2m  (m+1)^{\ell-2}  m^{r-\ell}&, \text{ otherwise.}  \\
         \end{cases}
\end{align*}
\begin{proof}[\bf Proof of Theorem~\ref{conjKatona2}]
Let $G$ be an $n$-vertex graph with $t_r(n)+1$ edges satisfying $\tau_{r+1}(G)\geq 2$.
By applying Theorem~\ref{propositionmainclique} for $s=2$ and $t=1$, we can assume without loss of generality that $G\in \mathcal{H}$ and thus has a vertex partition $V(G)=V_1 \cup V_2 \cup \ldots \cup V_r$ where $|V_1|\geq|V_2|\geq \ldots \geq |V_r|$ with exactly two class-edges $e_1,e_2$ and at most one missing cross-edge. We have
\begin{align}
\label{classsizesedges}
    \sum_{1 \leq i<j\leq r} |V_i||V_j| -e^c(V_1,V_2,\ldots,V_r)+2=e(G)=t_r(n)+1.
\end{align}
This means we have two cases: 
\begin{align*}e^c(V_1,V_2,\ldots,V_r)=0\quad \quad \quad \text{or} \quad \quad \quad e^c(V_1,V_2,\ldots,V_r)=1.
\end{align*}
{\bf Case 1: $e^c(V_1,V_2,\ldots,V_r)=0$}.\\
In this case 
\begin{align*}
    \sum_{1 \leq i<j\leq r} |V_i||V_j|=t_r(n)-1
\end{align*}
by \eqref{classsizesedges}. Note that $|V_1|-|V_r|=2$, because if $|V_1|-|V_r|\geq 3$, then moving one vertex from $V_1$ to $V_r$ increases the sum $\sum_{1 \leq i<j\leq r} |V_i||V_j|$ by at least two and thus 
\begin{align}
\label{ineqaulitysumclasses}
    \sum_{1 \leq i<j\leq r} |V_i||V_j|\leq t_r(n)-2,
\end{align}
a contradiction. If $|V_1|-|V_r|\leq 1,$ then $\sum_{1 \leq i<j\leq r} |V_i||V_j|=t_r(n)$, giving a contradiction as well. Thus, we have $|V_1|-|V_r|=2$. Similarly, $|V_2|-|V_{r-1}|\leq1$, as otherwise  \eqref{ineqaulitysumclasses} holds. Conclusively, there are only the following types of combination of class sizes. We list them depending on $\ell=n \mod r$. If $\ell\in\{0,1\}$, then
\begin{itemize}
    \item $|V_1|=\ldots=|V_{\ell+1}|= m+1$, $|V_{\ell+2}|=\ldots = |V_{r-1}|=m$, $|V_r|=m-1$. 
\end{itemize}
If $\ell=2$, then
\begin{itemize}
    \item Type 1: $|V_1|=m+2$, $|V_{2}|=\ldots = |V_{r}|=m$, 
    \item Type 2: $|V_1|=|V_2|=|V_{3}|= m+1$, $|V_{4}|=\ldots = |V_{r-1}|=m$, $|V_r|=m-1$. 
\end{itemize}
If $3\leq \ell \leq r-3$, then
\begin{itemize}
    \item Type 1: $|V_1|=m+2$, $|V_2|=\ldots=|V_{\ell-1}|=m+1$, $|V_{\ell}|=\ldots = |V_{r}|=m$, 
    \item Type 2: $|V_1|=\ldots=|V_{\ell+1}|= m+1$, $|V_{\ell+2}|=\ldots = |V_{r-1}|=m$, $|V_r|=m-1$. 
\end{itemize}
If $\ell=r-2$, then
\begin{itemize}
    \item Type 1: $|V_1|=m+2$, $|V_2|=\ldots=|V_{r-3}|=m+1$, $|V_{r-2}|=\ldots = |V_{r}|=m$, 
    \item Type 2: $|V_1|=\ldots=|V_{r-1}|= m+1$,  $|V_r|=m-1$. 
\end{itemize}
If $\ell=r-1$, then
\begin{itemize}
    \item $|V_1|=m+2$, $|V_2|=\ldots=|V_{r-2}|=m+1$, $|V_{r-1}|= |V_{r}|=m$. 
\end{itemize}
Denote the two class-edges $e_1,e_2$. Let $1\leq \alpha \leq \beta \leq r$ such that $e_1\subset V_\alpha$ and $e_2\subset V_\beta$.  The number of copies of $K_{r+1}$ containing $e_1$ but not $e_2$ is $\prod_{i\neq \alpha} |V_i|$. Similarly, the number of copies of $K_{r+1}$ containing $e_2$ but not $e_1$ is $\prod_{i\neq \beta} |V_i|$. Thus, the total number of copies of $K_{r+1}$ in $G$ is at least 
\begin{align}
\label{classsizesproduct}
    T_{r+1}(G)\geq \prod_{i\neq \alpha} |V_i| + \prod_{i\neq \beta} |V_i| \geq 2 \prod_{i=2}^r |V_i|.
\end{align}
We will now check that for any possible combination of class sizes the right hand side in \eqref{classsizesproduct} is at least $T_{r+1}(G_3)$. We distinguish cases depending on $\ell=n \mod r$.
If $\ell=0$, then
\begin{align*}
    2 \prod_{i=2}^r |V_i|= 2 m^{r-2}(m-1)\geq (2m-2)m^{r-2}=T_{r+1}(G_3).
\end{align*}
If $\ell=1$, then
\begin{align*}
    2 \prod_{i=2}^r |V_i|= 2 (m+1)m^{r-3}(m-1) \geq (2m-1)m^{r-2}=T_{r+1}(G_3).
\end{align*}
If $\ell=2$, then
\begin{align*}
    2 \prod_{i=2}^r |V_i|= \begin{cases}
    2m^{r-1},& \text{for Type 1}\\
     2 (m+1)^2m^{r-4}(m-1) ,& \text{for Type 2}\end{cases}\geq 2m^{r-1}=T_{r+1}(G_3).
\end{align*}
If $3\leq \ell \leq r-3$, then
\begin{align*}
    2 \prod_{i=2}^r |V_i|&= \begin{cases}
    2(m+1)^{\ell-2} m^{r-\ell+1},& \text{for Type 1}\\
     2 (m+1)^\ell m^{r-\ell-2}(m-1) ,& \text{for Type 2}\end{cases}\geq 2m^{r-\ell+1}(m+1)^{\ell-2}\\
     &=T_{r+1}(G_3).
\end{align*}
If $\ell=r-2$, then
\begin{align*}
        2 \prod_{i=2}^r |V_i|&= \begin{cases}
    2(m+1)^{r-4} m^{3},& \text{for Type 1}\\
     2 (m+1)^{r-2} (m-1) ,& \text{for Type 2}\end{cases}\geq 2m^{3}(m+1)^{r-4}\\
     &=T_{r+1}(G_3).
\end{align*}
Finally, if $\ell=r-1$, then
\begin{align*}
    2 \prod_{i=2}^r |V_i|= 
    2(m+1)^{r-3} m^{2}=T_{r+1}(G_3).
\end{align*}
Thus, in Case 1, $T_{r+1}(G)\geq T_{r+1}(G_3)$.

{\bf Case 2: $e^c(V_1,V_2,\ldots,V_r)=1$}.\\
Observe that in this case the class sizes of $G$ and $G_3$ are the same. Further, the number of missing cross-edges is exactly one. Now, we distinguish two cases depending on where $e_1$ and $e_2$ are. Case $2$a is when both class-edges are inside the same class. In Case $2$b, $e_1$ and $e_2$ are in different classes. \\
{\bf Case 2a:} $e_1,e_2\subseteq V_\alpha$ for some $1\leq \alpha \leq r.$ \\
The missing edge is incident to $e_1$ or $e_2$. The second endpoint of the missing edge is in $V_\beta$ for some $\beta\neq \alpha$. Then
\begin{align*}
T_{r+1}(G) &= 2\prod_{i\neq \alpha} |V_i|- \prod_{i\neq \alpha,\beta}|V_i| = (2|V_{\beta}|-1) \prod_{i\neq \alpha,\beta}|V_i|\\
&\geq  (|V_\alpha|+|V_{\beta}|-2) \prod_{i\neq \alpha,\beta}|V_i| \geq (|V_1|+|V_2|-2) \prod_{i=3}^r |V_i|= T_{r+1}(G_3), 
\end{align*}
where the sum of the factors are the same in the two products, hence the product is larger when the factors are `closer' to each other.\\
\noindent
{\bf Case 2b:} $e_1\subseteq V_\alpha, e_2 \subseteq V_\beta$ for some $\alpha \neq \beta$. \\
Since the missing edge is incident to both class-edges,
\begin{align*}
    T_{r+1}(G)&= (|V_\beta| -1) \prod_{i\neq \alpha,\beta }|V_i| + (|V_\alpha| -1) \prod_{i\neq \alpha,\beta }|V_i| \\ &=  (|V_\alpha|+|V_\beta| -2) \prod_{i\neq \alpha,\beta }|V_i| \geq (|V_1|+|V_2| -2) \prod_{i=3}^r |V_i|=T_{r+1}(G_3), 
\end{align*}
where the last inequality holds by the same reasoning as in Case $2$a.

\end{proof}

\section*{Acknowledgement}
We would like to thank an anonymous referee for helpful comments, in particular, for pointing out an error in an earlier version of this paper.

\bibliographystyle{abbrv}
\bibliography{Katona}

\end{document}